\documentclass[11pt]{amsart}
\usepackage{amsmath,amsthm,amssymb}
\theoremstyle{definition}
\newtheorem{dfn}{Definition}
\newtheorem{thm}[dfn]{Theorem}

\newtheorem{lem}[dfn]{Lemma}

\title{Fields whose Multiplicative Groups are Linear Spaces}
\author[Y.\ Nakata]{Yuki Nakata}
\address{Faculty of Science, Kyoto University}
\email{nakata.yuuki.42x@st.kyoto-u.ac.jp}

\begin{document}
\maketitle
\begin{abstract}
The purpose of this paper is to study fields whose multiplicative groups admit
the structure of linear spaces.
We prove that the multiplicative group of a finite field is a linear space 
if and only if the order of the multiplicative group is 1, 2, or a Mersenne prime.
We give necessary conditions 
for the multiplicative group of an infinite field  
to be a linear space over another field.
We also construct an example of an infinite field 
whose multiplicative group is a linear space over $\mathbb{Q}$.
\end{abstract}

\section{Introduction and Main Results}
The additive group of a field is a linear space over the prime field.
On the contrary, satisfactory characterization of the multiplicative groups of fields has not been given \cite[pp.\,704--705]{fuchs}, \cite{may}.
The purpose of this paper is to study fields whose multiplicative groups admit
the structure of linear spaces.

In this paper, all fields are commutative.
The multiplicative group of a field $K$ is denoted by $K^{\times}$.
We know the structure of multiplicative groups of typical fields 
such as finite fields, algebraic number fields, 
algebraically closed fields, 
real closed fields,
and $p$-adic number fields.
They are not linear spaces except finite fields~\cite[pp.\,701-704]{fuchs}:
\begin{itemize}
\item
$\mathbb{F}_q^{\times}\cong\mathbb{Z}/(q-1)\mathbb{Z}$,
where $q=p^n$, $p$ is a prime, and $n\in\mathbb{Z}_{>0}$.
\item
$\mathbb{Q}^{\times}\cong
\mathbb{Z}/2\mathbb{Z}\times\mathbb{Z}^{\oplus\aleph_0}$.
\item
$\mathbb{Q}(\alpha)^{\times}\cong
\mathbb{Z}/2m\mathbb{Z}\times\mathbb{Z}^{\oplus\aleph_0}$,
where $\alpha$ is algebraic over $\mathbb{Q}$, and $m\in\mathbb{Z}_{>0}$.
\item
If $K$ is an algebraically closed field (e.g.\ $\mathbb{C}$), we have
\[
K^{\times}\cong
\begin{cases}
\mathbb{Q}/\mathbb{Z}\times\mathbb{Q}^{\oplus\mathfrak{m}}
 & (\textrm{char}\, K=0)\\
\bigoplus_{\ell\neq p}\mathbb{Z}(\ell^{\infty})\times
\mathbb{Q}^{\oplus\mathfrak{m}}
& (\textrm{char}\, K=p>0),
\end{cases}
\]
where $\mathfrak{m}$ is a cardinal, $\ell$ is a prime, and
$\mathbb{Z}(\ell^{\infty}):=\mathbb{Z}[1/\ell]/\mathbb{Z}$ is the group of type
$\ell^{\infty}$.
\item
If $K$ is a real closed field (e.g.\ $\mathbb{R}$), we have
\[
K^{\times}\cong
\mathbb{Z}/2\mathbb{Z}\times\mathbb{Q}^{\oplus\mathfrak{m}},
\]
where $\mathfrak{m}$ is a cardinal.
\item
Let $p$ be a prime, and $\mathbb{Q}_{p}$ the $p$-adic number field.
Then we have
\[\mathbb{Q}_p^{\times}\cong\begin{cases}
\mathbb{Z}\times\mathbb{Z}/2\mathbb{Z}\times\mathbb{Z}_2 & (p=2) \\
\mathbb{Z}\times\mathbb{Z}/(p-1)\mathbb{Z}\times\mathbb{Z}_p&(p>2).
\end{cases}\]
\end{itemize}

In this paper, we give 
a necessary and sufficient condition
for the multiplicative group of a finite field to be a linear space,
give necessary conditions 
for the multiplicative group of an infinite field to be a linear space,
and construct an infinite field whose multiplicative group is a linear space.

The results are as follows.

\begin{thm}
\label{fin}
Let $p$ be a prime, and $q=p^n$ a power of $p$.
Let $\mathbb{F}_q$ be a finite field with $q$ elements.
The multiplicative group $\mathbb{F}_q^{\times}$ is a linear space 
if and only if $q=2, 3$, or $q-1$ is a Mersenne prime. 
(Recall that a prime $p'$ is a Mersenne prime if $p'=2^r-1$ for an integer $r$.)
\end{thm}

\begin{thm}
\label{infn}
Let $K$ be a field and $L$ be an infinite field.
\begin{enumerate}
\item
If the multiplicative group $L^{\times}$ is a linear space over $K$,
then the characteristic of $K$ is 0,
so that $L^{\times}$ is a linear space over $\mathbb{Q}$.
\item
If the multiplicative group $L^{\times}$ is a linear space over $\mathbb{Q}$, then
the characteristic of $L$ is 2,
and every element in $L$ except 0, 1 is transcendental over $\mathbb{F}_2$,
where we regard $\mathbb{F}_2$ as the prime field of $L$.
\end{enumerate}
\end{thm}

There exists an infinite field $L$ whose multiplicative group is a linear space over 
$\mathbb{Q}$.

\begin{thm}
\label{infc}
Let $\mathbb{F}_2((x))$ be the power series field (Laurent expansion field)
over $\mathbb{F}_2$.
Take a sequence $\{x^{1/n}\}_{n\in\mathbb{Z}_{>0}}$ 
of elements in an algebraic closure of $\mathbb{F}_2((x))$
satisfying
$(x^{1/mn})^m=x^{1/n}$ for all $m,n\in\mathbb{Z}_{>0}$ and $x^{1/1}=x^1=x$.
Let
\[
L_0:=\bigcup_{n=1}^{\infty}\mathbb{F}_2((x))(x^{1/n})
=\bigcup_{n=1}^{\infty}\mathbb{F}_2((x^{1/n}))
\]
be the extension of $\mathbb{F}_2((x))$ generated by $x^{1/n}$ for all $n\in\mathbb{Z}_{>0}$.
Then
the multiplicative group $L_0^{\times}$ is a linear space over $\mathbb{Q}$.
\end{thm}

We prove these theorems in the following sections.

\section{Proof of Theorem \ref{fin}}
\label{s_fin}
Obviously,
$\mathbb{F}_2^{\times}=\{1\}$ is the group with order 1, 
which is a linear space of dimension 0 over any field.

Assume $q\geq 3$.
Then
$\mathbb{F}_q^{\times}$ is a finite group whose order is larger than 1,
which cannot be a linear space over $\mathbb{Q}$.
Therefore $\mathbb{F}_q^{\times}\cong\mathbb{Z}/(q-1)\mathbb{Z}$ is a linear space over
$\mathbb{F}_{p'}$ ($p'$ is a prime) if and only if
\[
\mathbb{Z}/(q-1)\mathbb{Z}\cong(\mathbb{Z}/p'\mathbb{Z})^{\oplus m}
\]
for some $m\in\mathbb{Z}_{>0}$.
Since $\mathbb{Z}/(q-1)\mathbb{Z}$ is cyclic, we have $m=1$.
Thus we have $q-1=p'$.

If $q$ is even, then $q=2^n$ for some $n$.  
Hence $q-1$ is a Mersenne prime.

If $q$ is odd, then $q-1=2$.

\section{Proof of Theorem \ref{infn}}
\label{s_infn}

\begin{lem}
\label{main2}
Let $K,L$ be fields.
Assume that $L^{\times}$ is a linear space over $K$.
Then the characteristic of $K$ or $L$ is 2.
\end{lem}
\begin{proof}
Let us express a scalar product as an exponent like 
$a^r$ for $a\in L^{\times}$ and $r\in K$
for the compatibility with the multiplicative notations.

Let $S\subset L^{\times}$ be a basis of $L^{\times}$ as a linear space over $K$.
Assume that the characteristic of $L$ is not 2.
Then we have $-1\neq 1$ in $L$.
There exist
$k\in\mathbb{Z}_{>0}$, 
$a_1,\ldots, a_k\in S$, and $r_1,\ldots,r_k\in K\setminus\{0\}$ such that
$a_1^{r_1}\cdots a_k^{r_k}=-1$.
Then we have
\[
1=(-1)^2=a_1^{2r_1}\cdots a_k^{2r_k}.
\]
Since the elements of the basis are linearly independent,
we have $2r_1=\cdots=2r_k=0$.
Hence the characteristic of $K$ is 2.
(Note that 1 is the zero vector in the linear space $L^{\times}$.)
\end{proof}

We return to the proof of Theorem \ref{infn}.

(1)
Assume that the characteristic of $K$ is $p>0$.
Then we have $a^p=1$ for all $a\in L^{\times}$.
Thus every $a\in L^{\times}$ is a root of the polynomial $T^p-1$,
which implies $L^{\times}$ is a finite group, 
contradicting the assumption that $L$ is an infinite field.

(2)
By Lemma \ref{main2},
the characteristic of $L$ is 2.
Assume that $a\in L\setminus \{0,1\}$ is algebraic over $\mathbb{F}_2$.
Then $\mathbb{F}_2(a)$ is a finite field, 
and $\mathbb{F}_2(a)^{\times}$ is a finite group.
Then there exists a positive integer $n$ such that $a^n=1$, 
contradicting the assumption that 
$L^{\times}$ is a linear space over $\mathbb{Q}$.

\section{Proof of Theorem \ref{infc}}
\label{s_infc}

\begin{lem}
\label{root}
Let $\mathbb{F}_2[[x]]$ be the formal power series ring over $\mathbb{F}_2$.
The map 
\[
(-)^k\colon \mathbb{F}_2[[x]]^{\times}\to\mathbb{F}_2[[x]]^{\times},\ 
a(x)\mapsto a(x)^k\]
is a group isomorphism
for any positive odd integer $k$.
\end{lem}
\begin{proof}
The map
$(-)^k$ is a well-defined group homomorphism as the multiplication is commutative.
We shall prove that for any $a(x)\in\mathbb{F}_2[[x]]^{\times}$,
there exists a unique $b(x)\in\mathbb{F}_2[[x]]^{\times}$ such that
$a(x)=b(x)^k$.
Note that
$(\mathbb{F}_2[[x]],\ x\mathbb{F}_2[[x]])$ is a complete local ring.
For a polynomial 
\[
P(T)=a_0(x)+a_1(x)T+\cdots +a_m(x)T^m\in\mathbb{F}_2[[x]][T]\]
with $a_i(x)\in\mathbb{F}_2[[x]]$
for all $i\geq 0$,
we put
\[\overline{P}(T)=a_0(0)+a_1(0)T+\cdots +a_m(0)T^m
\in\mathbb{F}_2[T].\]

Take an element $a(x)\in\mathbb{F}_2[[x]]^{\times}$.
Consider the polynomial 
\[F(T):=T^k-a(x)\in\mathbb{F}_2[[x]][T].\]
Since $a(0)\in\mathbb{F}_2^{\times}=\{1\}$, we have
\[\overline{F}(T)=T^k-a(0)=(T-1)(T^{k-1}+\cdots +1).\] 
Since $k$ is odd, the polynomials
$T-1$ and $T^{k-1}+\cdots +1\in\mathbb{F}_2[T]$ are relatively prime.
By Hensel's Lemma \cite[Theorem 8.3]{matsu}, we have 
\[F(T)=(T-b(x))H(T)\] 
for some $b(x)\in\mathbb{F}_2[[x]]$ and $H(T)\in\mathbb{F}_2[[x]][T]$
satisfying $b(0)=1$ and
$\overline{H}(T)=T^{k-1}+\cdots +1$.

We shall prove $H(T)$ has no root in $\mathbb{F}_2[[x]]$.
Assume that $H(T)$ has a root in $\mathbb{F}_2[[x]]$.
Then we have $H(c(x))=0$ for some $c(x)\in\mathbb{F}_2[[x]]$. 
Putting $x=0$, we have
$\overline{H}(c(0))=0$,
which contradicts $\overline{H}(0)=\overline{H}(1)=1$
as $k$ is odd.

Therefore the factorization $T^k-a(x)=(T-b(x))H(T)$
implies that $b(x)$ is a unique $k$-th root of $a(x)$ in $\mathbb{F}_2[[x]]^{\times}$.
\end{proof}

In the following, 
we take a sequence $\{x^{1/n}\}_{n\in\mathbb{Z}_{>0}}$ as in Theorem \ref{infc}.
We consider the ring
\[R:=\displaystyle\bigcup_{n=1}^{\infty}\mathbb{F}_2[[x^{1/n}]]\subset L_0.\]

\begin{lem}
\label{ring}
The unit group
\[
R^{\times}=\bigcup_{n=1}^{\infty}\mathbb{F}_2[[x^{1/n}]]^{\times}
\]
is a linear space over $\mathbb{Q}$.
\end{lem}
\begin{proof}
It suffices to prove that 
\[(-)^k\colon R^{\times}\to R^{\times},\ a(x)\mapsto a(x)^k\] is a group isomorphism
for any positive integer $k$.
It is enough to prove the assertion in the case of odd $k$ and $k=2$ separately.

If $k$ is odd, the map
\[(-)^k\colon\mathbb{F}_2[[x^{1/n}]]^{\times}\to\mathbb{F}_2[[x^{1/n}]]^{\times},\ 
a(x)\mapsto a(x)^k\] 
is a group isomorphism
by Lemma \ref{root}.
Hence
the extension map $(-)^k\colon R^{\times}\to R^{\times}$ to the union
is also a group isomorphism.

If $k=2$, the map $(-)^2$ is the restriction of the Frobenius map 
$L_0\to L_0$ sending $a(x)$ to $a(x)^2$. 
Then we see that this map is injective. 
This map is surjective because
\[
(-)^{1/2}\colon R^{\times}\to R^{\times}\] 
defined by 
\[(1+c_{1/n}x^{1/n}+c_{2/n}x^{2/n}+\cdots)\ \mapsto\  
(1+c_{1/n}x^{1/2n}+c_{2/n}x^{2/2n}+\cdots)
\]
is the inverse map of $(-)^2$.
\end{proof}

Finally we return to the proof of Theorem \ref{infc}.
Note that $\mathbb{F}_2[[x^{1/n}]]$ is a discrete valuation ring with uniformizer $x^{1/n}$. 
Thus we have
\[
\mathbb{F}_2((x^{1/n}))^{\times}
=\coprod_{k=-\infty}^{\infty}
x^{k/n}\mathbb{F}_2[[x^{1/n}]]^{\times}.
\]
Therefore we have
\begin{align*}
L_0^{\times}&=\bigcup_{n=1}^{\infty}\left(\coprod_{k=-\infty}^{\infty}
x^{k/n}\mathbb{F}_2[[x^{1/n}]]^{\times}\right)\\
&=\coprod_{\alpha\in\mathbb{Q}}x^{\alpha}\left(\bigcup_{n=1}^{\infty}
\mathbb{F}_2[[x^{1/n}]]^{\times}\right)\\
&=\coprod_{\alpha\in\mathbb{Q}}x^{\alpha}R^{\times}.
\end{align*}
Here we put $x^{\alpha}R^{\times}:=\{x^{\alpha}a(x)\ |\ a(x)\in R^{\times}\}$.
Thus we have an isomorphism 
\[
L_0^{\times}\cong\mathbb{Q}\times R^{\times}.
\]
Since $R^{\times}$ is a linear space over $\mathbb{Q}$ by Lemma \ref{ring},
we conclude that $L_0^{\times}$ is a linear space over $\mathbb{Q}$.

\subsection*{Acknowledgements}
The content of this paper is based on the author's presentation 
at the 8th Science Intercollegiate Contest
held on March 2-3, 2019, hosted by 
Ministry of Education, Culture, Sports, Science and Technology (MEXT) of Japan.
The author would like to thank the organizers of the contest 
for providing the opportunity to present the author's research.

\end{document}